\documentclass[11pt,a4paper,british]{amsart}

\usepackage{amsthm,amsmath,amscd}
\usepackage[T1]{fontenc}
\usepackage{babel}
\usepackage{microtype}
\usepackage{thmtools}
\usepackage[dvipsnames]{xcolor}
\definecolor{fernuni_lightblue}{RGB}{0,113,185}
\definecolor{strongred}{RGB}{160,0,0}

\usepackage{dutchcal}
\usepackage{mlmodern}
\usepackage{dsfont}
\usepackage{mathtools}
\let\mathbb\mathds

\colorlet{background}{white}
\colorlet{text}{black}
\colorlet{displayed}{black}
\colorlet{bold}{black}
\colorlet{url}{fernuni_lightblue}
\colorlet{italic}{black}
\colorlet{link}{strongred}
\pagecolor{background}
\color{text}

\makeatletter
\def\@setauthors{%
  \begingroup
  \def\thanks{\protect\thanks@warning}%
  \trivlist
  \centering\footnotesize \@topsep30\p@\relax
  \advance\@topsep by -\baselineskip
  \item\relax
  \author@andify\authors
  \def\\{\protect\linebreak}%
  \sc\Large{\MakeLowercase{\authors}}%
  \ifx\@empty\contribs
  \else
    ,\penalty-3 \space \@setcontribs
    \@closetoccontribs
  \fi
  \endtrivlist
  \endgroup
}
\ifdefined\@secnumfont\renewcommand{\@secnumfont}{\bfseries}\fi
\renewcommand{\section}{\@startsection{section}{1}{0mm}{-1\baselineskip}{0.8\baselineskip}{\color{bold}\bfseries\large}}
\renewcommand{\subsection}{\@startsection{subsection}{2}{0mm}{-1\baselineskip}{0.6\baselineskip}{\color{bold}\bfseries\normalsize}}
\renewcommand{\subsubsection}{\@startsection{subsubsection}{2}{0mm}{-1\baselineskip}{0.4\baselineskip}{\color{bold}\bfseries\normalsize}}

\usepackage[marginratio=1:1,paperwidth=460pt,paperheight=732pt,height=632pt,width=360pt]{geometry}

\usepackage{amssymb}
\usepackage{mathrsfs}

\usepackage{bm}

\usepackage[colorlinks=false,backref=page]{hyperref}
\hypersetup{pdfborder = {0 0 0}}
\let\newhref\href

\def\@cite#1#2{[\textcolor{link}{\textbf{#1}}\if@tempswa , #2\fi]}
\renewcommand{\href}[2]{\newhref{#1}{\color{url}{#2}}}
\renewcommand{\eqref}[1]{(\textcolor{link}{\ref{#1}})}

\newtheoremstyle{plain}
  {8pt}{8pt}{\itshape\color{italic}}{}{\bfseries\color{bold}}{.}{0.5em}{}
\newtheoremstyle{definition}
  {8pt}{8pt}{}{}{\bfseries\color{bold}}{.}{0.5em}{}
\newtheoremstyle{remark}
  {8pt}{8pt}{}{}{\bfseries\color{bold}}{.}{0.5em}{}

\pagecolor{background}
\color{text}
\DeclareRobustCommand{\itshape}{\not@math@alphabet\itshape\mathit\fontshape\itdefault\selectfont\color{italic}}
\renewcommand\emph[1]{\textit{\color{italic}#1}}
\def\@settitle{\begin{center}%
  \baselineskip14\p@\relax
    \bfseries\Large\color{bold}
  \@title
  \end{center}%
}
\usepackage{etoolbox}
\AtBeginEnvironment{equation}{\color{displayed}}
\AtBeginEnvironment{equation*}{\color{displayed}}
\AtBeginEnvironment{align}{\color{displayed}}
\AtBeginEnvironment{align*}{\color{displayed}}
\AtBeginEnvironment{gather}{\color{displayed}}
\AtBeginEnvironment{gather*}{\color{displayed}}
\AtBeginEnvironment{multline}{\color{displayed}}
\AtBeginEnvironment{multline*}{\color{displayed}}
\makeatother

\usepackage[capitalize,noabbrev]{cleveref}

\renewcommand{\cref}[1]{%
  \namecref{#1}~\textcolor{link}{\labelcref{#1}}%
}

\renewcommand*{\d}{\mathrm{d}}

\newcommand*{\R}{\mathbb{R}}

\newcommand*{\tr}{\operatorname{tr}}

\newcommand*{\e}{\mathrm{e}}
\newcommand*{\grad}{\operatorname{grad}}

\newtheorem{thm}{Theorem}[section]
\newtheorem{cor}[thm]{Corollary}

\newtheorem{prop}[thm]{Proposition}
\newtheorem{lem}[thm]{Lemma}

\theoremstyle{remark}
\newtheorem{rmk}[thm]{Remark}

\theoremstyle{definition}
\newtheorem{defn}[thm]{Definition}

\numberwithin{equation}{section}

\usepackage{thmtools,thm-restate,comment}

\date{2nd December 2026}

\numberwithin{equation}{section}

\newcommand{\weyl}{{}^{[g]}\nabla}
\newcommand{\higgs}{A_{[g]}}
\newcommand{\levi}{{}^g\nabla}
\newcommand{\llangle}{\langle\!\langle}
\newcommand{\rrangle}{\rangle\!\rangle}

\newcommand{\pvar}{\left.\tfrac{\partial}{\partial t}\right|_{t=0}}

\newcommand{\mal}{\circledast_g\!}

\renewcommand*{\d}{\mathrm{d}}

\renewcommand*{\div}{\operatorname{div}}
\newcommand*{\sym}{\operatorname{Sym}}
\newcommand*{\Vol}{\operatorname{Vol}}
\newcommand*{\Ric}{\operatorname{Ric}}
\newcommand{\proj}{\mathfrak{p}}
\newcommand{\tg}{\tilde{g}}
\newcommand{\loc}{\Sigma_{[g]}}

\DeclareMathOperator{\con}{Con}
\newcommand{\Riem}{\mathrm{Riem}(M)}

\title[A Variational Problem for Riemannian Metrics]{Prescribing geodesics and a variational problem for Riemannian metrics}
\author[T.~Mettler]{Thomas Mettler}
\address{Faculty of Mathematics and Computer Science, UniDistance Suisse, Brig, Switzerland}
\email{thomas.mettler@fernuni.ch; mettler@math.ch}

\begin{document}

\begin{abstract}
Given a prescription of unparametrised paths on a manifold $M$, one path for each tangent direction, we may ask whether these paths agree with the geodesics of a Riemannian metric on $M$. Generically, this is not the case. Motivated by this fact, we introduce a non-negative functional $\mathcal{E}$ on the space of Riemannian metrics on $M$ so that $\mathcal{E}(g)=0$ if and only if the geodesics of the metric $g$ agree with the prescribed paths. We compute the variational equations for $\mathcal{E}$ and show that the conformal variational equation is, perhaps surprisingly, of Yamabe type. This allows us to obtain existence results for conformally critical points of $\mathcal{E}$. In particular, in the surface case, every conformal class contains a conformally critical metric, unique up to homothety. In dimension at least three, $\mathcal{E}$ is, up to a constant factor, the Einstein--Hilbert functional of a suitably defined projective scalar curvature, and its critical metrics are characterised by a projective analogue of the vacuum Einstein equations. As a by-product, we establish that the Blaschke metric of a properly convex projective surface is a critical point for $\mathcal{E}$.  
\end{abstract}

\maketitle

\section{Introduction}

A \emph{projective structure} on a manifold $M$ is an equivalence class $\proj$ of torsion-free connections on its tangent bundle $TM$, where two connections are equivalent if they have the same geodesics up to parametrisation. Here by a geodesic for $\nabla \in \proj$ we mean an immersed curve $\gamma : I \to M$ so that $\nabla_{\dot{\gamma}}\dot{\gamma}$ vanishes identically. A projective structure $\proj$ thus prescribes an unparametrised path through every point of $M$ in every tangent direction. We call $\proj$ \emph{metrisable} if the prescribed paths are the geodesics of some Riemannian metric $g$ on $M$, that is, if its Levi-Civita connection $\levi$ is an element of $\proj$. Necessary and sufficient conditions for metrisability of a two-dimensional projective structure were obtained by Bryant, Dunajski and Eastwood \cite{MR2581355}; in particular, a generic projective structure is not metrisable.  It is therefore natural to look for a non-negative functional on the space $\mathrm{Riem}(M)$ of Riemannian metrics on $M$ which vanishes precisely on those metrics whose geodesics agree with the prescribed paths.

 For a closed oriented projective manifold $(M,\proj)$ of dimension $n\geqslant 2$, we introduce such a functional $\mathcal{E}$ as follows. Given a representative $\nabla \in \proj$, the difference $\nabla-\levi$ is a $1$-form on $M$ with values in the endomorphisms of $TM$; we write $(\cdot)_0$ for the part of such a $1$-form taking values in the trace-free endomorphisms of $TM$ and define
\[
	\mathcal{E} : \mathrm{Riem}(M) \to \R, \qquad g \mapsto \Vol_{M}(g)^{(2-n)/n}\int_{M} |(\nabla-\levi)_0|^2_g d\mu_g,
\]
where $d\mu_g$ denotes the volume form of $g$, $\Vol_M(g)$ the volume of $M$ with respect to $d\mu_g$ and $|\,\cdot\,|_g$ the pointwise norm induced by $g$ on the fibres of $S^2(T^*M)\otimes TM$. It turns out that the non-negative functional $\mathcal{E}$ depends only on the projective equivalence class of $\nabla$ and, moreover, that $\mathcal{E}(g)=0$ for $g \in \Riem$ if and only if $\proj$ is metrisable by the metric $g$. Geometrically, the tensor field $\Phi_g:=(\nabla-\levi)_0$ determines the rate at which the geodesics of $\proj$ spread away from those of $g$; see \cref{sec:spreading}. The presence of the factor $\Vol_{M}(g)^{(2-n)/n}$ makes $\mathcal{E}$ invariant under homothety, that is, $\mathcal{E}(cg)=\mathcal{E}(g)$ for all positive constants $c$ and all metrics $g \in \Riem$.

In \cref{prop:vareq}, we show that the variational equations for $\mathcal{E}$ take the form 
\[
\mathsf{q}_g(\Phi_g)-2\ell^*_g(\Phi_g)=\left(\frac{n-2}{2n}\right)\overline{|\Phi_g|^2_g}\,g,
\]
where $\overline{|\Phi_g|^2_g}$ denotes the mean value of $|\Phi_g|^2_g$ with respect to $d\mu_g$, and $\mathsf{q}_g(\Phi_g) \in \Gamma(S^2(T^*M))$ is a purely algebraic term in $\Phi_g$. The map $\ell^*_g$ is the formal adjoint of a first-order differential operator $\ell_g$ obtained from the linearisation of the map $\zeta : g\mapsto\levi$, which assigns to a metric its Levi-Civita connection.

An important class of (in general) non-metrisable projective surfaces carrying a distinguished metric -- the so-called \emph{Blaschke metric} -- is given by the properly convex projective surfaces. They are of particular interest due to their connection to higher Teichm\"uller theory. For background we refer to the landmark paper by Hitchin \cite{MR1174252}; see also \cite{MR3966798} for a recent survey. Given a Riemann surface structure on $M$, every holomorphic cubic differential $C$ determines a Riemannian metric $g$ -- the Blaschke metric -- via Hitchin's self-duality equations \cite{MR887284}, and the pair $(g,C)$ encodes a properly convex projective structure on $M$ \cite{MR1828223,MR2402597}. Alternatively, the Blaschke metric arises via a hyperbolic affine sphere over the universal cover of $M$, obtained from a Monge--Amp\`ere equation solved by Cheng--Yau \cite{MR0437805,MR859275}; see \cite{MR2743442} for a survey. Our first main result complements these pictures by relating the Blaschke metric to a variational principle intrinsic to the projective structure:

\begin{restatable}{mainthm}{maintheoremone}\label{thm:maintheoremone}
Let $(M,\proj)$ be a closed oriented properly convex projective surface of negative Euler characteristic. Then the Blaschke metric of $\proj$ is a critical point for $\mathcal{E}$.  
\end{restatable}

We now turn to conformal variations of $\mathcal{E}$. In two dimensions, solving the variational equations for conformal variations is equivalent to solving a Poisson equation, and this allows us to prove:

\begin{restatable}{mainthm}{maintheoremtwo}\label{thm:maintheoremtwo}
Let $(M,\mathfrak{p})$ be a closed oriented projective surface. Then every conformal structure on $M$ contains a conformally critical metric, unique up to homothety. 
\end{restatable}

A surprising aspect of $\mathcal{E}$ is its close connection to the Yamabe equation \cite{MR125546} and to the Einstein--Hilbert functional, which assigns to a metric $g$ its volume-normalised total scalar curvature $\Vol_{M}(g)^{(2-n)/n}\int_{M}S_g d\mu_g$. To explain this connection, we use \cite[Thm.~2.4]{MR4201191}, which shows that the choice of a conformal structure $[g]$ on a projective manifold $(M,\proj)$ distinguishes a unique torsion-free connection $\weyl$ preserving $[g]$ -- a so-called \emph{Weyl connection} -- and a $1$-form $\higgs$ on $M$ with values in the endomorphisms of $TM$ so that $\weyl+\higgs$ lies in $\proj$.

For a metric $g$ on $(M,\proj)$, we define the \emph{conformal defect} as
\[
	V_g=\tfrac{(n+1)(n-2)}{(n+2)}|\higgs|^2_g-\tr_g\Ric\left(\weyl\right),
\]
where $\Ric\left(\weyl\right)$ denotes the Ricci curvature of $\weyl$. The motivation for introducing $V_g$ is its relevance in finding conformally critical metrics for $\mathcal{E}$. Recall that, for $n\geqslant 3$ the conformal Laplacian of $g$ is defined as 
 \[
 L_g:=4\left(\frac{n-1}{n-2}\right)\Delta_g + S_g.
 \]
 Using $V_g$, we define the \emph{projective-conformal Laplacian} $\mathcal{L}_g=L_g+V_g$ and show that finding a metric in $[g]$ that is conformally critical for $\mathcal{E}$ amounts to solving an equation of Yamabe type
 \[
  \mathcal{L}_g u=\kappa u^{(n+2)/(n-2)}
 \]
 for some constant $\kappa$ and smooth positive function $u$ on $M$.

Under a conformal change $g \mapsto \tilde{g}=\exp(2f)g$, the conformal defect changes as $V_{\tg}=\exp(-2f)V_g$. As a consequence of this, it follows that the \emph{projective scalar curvature} 
\[
	\mathcal{S}_g:=S_g+V_g
\]
transforms exactly as the scalar curvature under a conformal change $g\mapsto \tg$. Furthermore, we have the identity 
\[
	\mathcal{E}(g)=\tfrac{n+2}{(n+1)(n-2)}\Vol_{M}(g)^{(2-n)/n}\int_{M}\mathcal{S}_g d\mu_g,
\]
showing that our functional closely resembles the Einstein--Hilbert functional. Defining the \emph{projective Einstein tensor} $\mathcal{G}_g$ by the same variational characterisation that defines the Einstein tensor, with the scalar curvature $S_g$ replaced by the projective scalar curvature $\mathcal{S}_g$ (see \cref{sub:projective_einstein}), a metric $g$ is critical for $\mathcal{E}$ if and only if, for $n\geqslant 3$, $\mathcal{G}_g+\Lambda g=0$ for a constant $\Lambda$ (see \cref{prop:projeinstein}). Unlike the usual Einstein tensor, the projective Einstein tensor $\mathcal{G}_g$ is in general not divergence-free, since $\mathcal{E}$ is invariant only under the diffeomorphisms of $M$ preserving $\proj$. 

We define the \emph{conformal locus} of $[g]$ as $\loc:=\left\{p \in M\,|\, V_g(p)=0\right\}$. Note that $\loc$ is a conformal invariant of $[g]$. Building on the solution \cite{MR240748,MR0433500,MR788292} of the Yamabe problem \cite{MR125546} and the theory of elliptic equations of Yamabe type (see in particular \cite{MR2148873}), we obtain:
\begin{restatable}{mainthm}{maintheoremthree}\label{thm:maintheoremthree}
Let $(M,\proj)$ be a closed oriented projective manifold of dimension $n\geqslant 3$ and $[g]$ a conformal structure on $M$. Suppose that either $\loc=M$, or that $n\geqslant 4$ and there exists $p_0\in M$ with $V_{\tg}(p_0)<0$ for some (and hence any) $\tg \in [g]$. Then $[g]$ contains a conformally critical metric for $\mathcal{E}$.
\end{restatable}

Drawing further on the analogy with the Yamabe equation, it is tempting to speculate that on a closed oriented projective manifold $(M,\proj)$, every conformal equivalence class $[g]$ contains a conformally critical metric for $\mathcal{E}$. Some of the cases left open by \cref{thm:maintheoremthree} appear to be within reach of the localisation theory for elliptic equations of Yamabe type developed by Druet and Hebey \cite{MR2148873}. Others -- most notably the borderline regime in which $V_g\geqslant 0$ with $\loc\neq\emptyset$ -- parallel the situation in which the classical Yamabe problem was settled by Schoen \cite{MR788292} via the positive mass theorem of Schoen and Yau \cite{MR526976,MR612249}. We hope to return to these unresolved cases in future work.

\subsection*{Acknowledgements} The author is grateful to Maciej Dunajski, Rod Gover, Luca Giudici and Tobias Weth for helpful conversations.

\subsection*{Statement on AI usage} AI tools were used for proofreading and minor language polishing. All mathematical arguments are due to the author.

\section{Setup}

\subsection{Projective structures}

Let $M$ be a closed oriented smooth $n$-manifold with $n\geqslant 2$. Recall that the difference of two torsion-free connections $\nabla$ and $\nabla^{\prime}$ on $TM$ is a section of the vector bundle $E=S^2(T^*M)\otimes TM$. For a section $\Phi=(\Phi^i_{jk})$ of $E$ we define $\con(\Phi) \in \Omega^1$ by the pointwise tensor contraction 
\[
[\con(\Phi)]_i=\Phi^k_{ik}.
\] 
For $\alpha=(\alpha_i) \in\Omega^1$ we obtain a section $\sym(\alpha)$ of $E$ by the rule
\[
\left[\sym(\alpha)\right]^i_{jk}=\delta^i_j\alpha_k+\delta^i_k\alpha_j. 
\]
Observe that $\con(\sym(\alpha))=(n+1)\alpha$.

For a section $\Phi$ of $E$ we denote by $\Phi_0$ its part in the kernel of the above contraction, that is,
\begin{equation}\label{eqn:glndecomp}
\Phi_0:=\Phi-\tfrac{1}{n+1}\sym\left(\con(\Phi)\right).
\end{equation}
In particular, we have
\[
\left(\sym(\alpha)\right)_0=\sym(\alpha)-\tfrac{1}{(n+1)}\sym(\con(\sym(\alpha)))=\sym(\alpha)-\sym(\alpha)=0
\]
for all $\alpha \in\Omega^1$, so that
\[
E\simeq E_0\oplus T^*M,
\]
where $E_0$ denotes the contraction-free part of $E$. 
\begin{defn}
Two torsion-free connections $\nabla$ and $\nabla^{\prime}$ on $TM$ are called \textit{projectively equivalent} if every geodesic of $\nabla$ can be reparametrised to become a geodesic of $\nabla^{\prime}$. A \textit{projective structure} on $M$ is an equivalence class $\mathfrak{p}$ of projectively equivalent torsion-free connections on $TM$. A manifold $M$ equipped with a projective structure $\proj$ will be called a \emph{projective manifold}.
\end{defn}

\begin{rmk}
By a geodesic of $\nabla$ we mean an immersed curve $\gamma : I \to M$ which satisfies $\nabla_{\dot\gamma}\dot\gamma=0$. 
\end{rmk}
We have -- see \cite{MR532831} for a modern reference:
\begin{lem}[Cartan, Eisenhart, Weyl]\label{lem:projequiv}
Two torsion-free connections $\nabla$ and $\nabla^{\prime}$ on $TM$ are projectively equivalent if and only if $(\nabla-\nabla^{\prime})_0=0$. 
\end{lem}

\subsection{A choice of metric}

Besides a projective structure $\proj$, we now also fix a Riemannian metric $g$ on $M$. The metric $g=(g_{ij})$ equips each real tensor bundle $E \to M$ with a bundle metric. Abusing notation, we denote this bundle metric on all the various tensor bundles by $\langle\cdot,\cdot\rangle_g$. For instance, if $\Phi=(\Phi^i_{jk})$ and $\Psi=(\Psi^i_{jk})$ are sections of $S^2(T^*M)\otimes TM$, then
\[
\langle\Phi,\Psi\rangle_g=\Phi^i_{jk}g^{jb}g^{kc}g_{ia}\Psi^a_{bc},
\]
where we write $g^{\sharp}=(g^{ij})$ for the dual metric. Correspondingly, we obtain an inner product on the set $\Gamma(E)$ of smooth sections of $E$ by defining
\[
\llangle s_1,s_2\rrangle_g=\int_M\langle s_1,s_2\rangle_g d\mu_g,
\]
where $d\mu_g$ denotes the volume form of $g$ and $s_1,s_2\in \Gamma(E)$. As usual, we write $|s|_g:=\sqrt{\langle s,s\rangle_g}$ and likewise $\Vert s\Vert_g:=\sqrt{\llangle s,s\rrangle_g}$ for all $s \in \Gamma(E)$. For what follows, we record the identities
\begin{equation}\label{eq:basicid1}
|\Phi|^2_{\exp(2f)g}=\e^{-2f}|\Phi|^2_g
\end{equation}
and
\begin{equation}\label{eq:basicid2}
d\mu_{\exp(2f)g}=\e^{nf}d\mu_g,
\end{equation}
where $f \in C^{\infty}$ and $\Phi \in \Gamma(S^2(T^*M)\otimes TM)$.  

The bundle $E_0$ with $E=S^2(T^*M)\otimes TM$ is irreducible as a $\mathrm{GL}^+(n,\R)$-bundle, it is however not irreducible as a $\mathrm{SO}(n)$-bundle. Consider the metric trace
\[
	\tr_g : \Gamma(E_0) \to \Gamma(TM), \quad \Phi^i_{jk} \mapsto g^{jk}\Phi^i_{jk} 
\]
and the inclusion
\[
	\iota_g : \Gamma(TM) \to \Gamma(E_0),\quad X \mapsto (g\otimes X)_0
\]
which satisfy
\[
	\tr_g \circ \iota_g = \tfrac{(n+2)(n-1)}{(n+1)}\mathrm{Id}_{\Gamma(TM)}.
\]
Consequently, we can decompose $\Phi \in \Gamma(E_0)$ as
\begin{equation}\label{eqn:decomp}
	X=\frac{1}{m} \tr_g \Phi\quad \text{and} \quad A=\Phi-(g\otimes X)_0,
\end{equation}
where $m=\tfrac{(n+2)(n-1)}{(n+1)}$. By construction, $\tr_gA$ vanishes identically. Furthermore, a simple computation shows that this decomposition is orthogonal. That is, if $Z$ is a vector field on $M$ and $\Phi$ a section of $E_0$ which satisfies $\tr_g\Phi\equiv 0$, then $\langle (g\otimes Z)_0,\Phi\rangle_g$ vanishes identically as well.
Moreover, for the decomposition \eqref{eqn:decomp}, we have
\begin{equation}\label{eqn:normsplit}
	|\Phi|^2_g=|A|^2_g+m|X|^2_g,
\end{equation}
since $|(g\otimes Y)_0|^2_g=m|Y|^2_g$ for every vector field $Y$ on $M$. 

\begin{rmk}
Also observe that for $\Psi \in \Gamma(E_0)$ and a $1$-form $\alpha$ on $M$ we have
\begin{equation}\label{eqn:sondecomp}
\langle \Psi,\sym(\alpha)\rangle_g=0,
\end{equation}
as can be verified with a simple computation.
\end{rmk}

\subsection{A projectively invariant functional}

Let $\mathrm{Riem}(M)$ denote the set of Riemannian metrics on $M$. For a projective structure $\mathfrak{p}$ and orientation on $M$ we consider the following projectively invariant non-negative functional
\[
\mathcal{E} : \mathrm{Riem}(M) \to \R, \qquad g \mapsto \Vol_{M}(g)^{(2-n)/n}\int_{M} |(\nabla-\levi)_0|^2_g d\mu_g, 
\]
where $\nabla \in \mathfrak{p}$ and $\levi$ denotes the Levi-Civita connection of $g$. Note that the functional attains the value $0$ if and only if there exists a Riemannian metric $g$ whose Levi-Civita connection belongs to $\mathfrak{p}$. Also observe that if $\hat{\mathcal{F}} : \mathrm{Riem}(M) \to \R$ is a functional on the set of Riemannian metrics for which there exists a constant $k$ such that $\hat{\mathcal{F}}(\e^{2c}g)=\e^{kc}\hat{\mathcal{F}}(g)$
for all constants $c$ and metrics $g$, then the functional $\mathcal{F}$ defined by $\mathcal{F}(g)=\Vol_{M}(g)^{-k/n}\hat{\mathcal{F}}(g)$ is invariant under rescaling a metric by a positive constant. Since rescaling a metric by a positive constant does not change its Levi-Civita connection, \eqref{eq:basicid1} and \eqref{eq:basicid2} imply that
\[
\mathcal{E}(cg)=\mathcal{E}(g)
\]
for all positive constants $c$ and all metrics $g \in \Riem$. 

\begin{rmk}[Notation]
The functional $\mathcal{E}$ depends on a choice of projective structure $\proj$ on $M$, as do a number of other objects we shall define. To keep the notation uncluttered, we leave this dependency implicit throughout.
\end{rmk}

\section{Variational equations}

Let $V$ be a Fr\'echet space and $U\subset V$ an open subset. The partial derivative of a function $\mathcal{F} : U \to \R$ at $g \in U$ in the direction of $h \in V$ is defined by
\[
\mathcal{F}^{\prime}_g(h):=\lim_{t\to 0}\frac{1}{t}(\mathcal{F}(g+th)-\mathcal{F}(g)),
\]
provided the limit exists. The set $\Gamma(S^2(T^*M))$ is a Fr\'echet space and the set of Riemannian metrics on $M$ is an open subset thereof, hence $\mathrm{Riem}(M)$ is a Fr\'echet manifold whose tangent space at $g$ is canonically isomorphic to $\Gamma(S^2(T^*M))$.

We say a Riemannian metric $g$ on $M$ is \textit{critical} for $\mathcal{E}$ provided $\mathcal{E}^{\prime}_g(h)=0$ for all $h \in T_g\mathrm{Riem}(M)\simeq \Gamma(S^2(T^*M))$. Likewise, we say a metric $g$ on $M$ is \textit{conformally critical} for $\mathcal{E}$, provided $\mathcal{E}^{\prime}_g(h)=0$ for all $h$ of the form $h=fg$ with $f \in C^{\infty}$. 

For a Riemannian metric $g$ on $M$, $h \in \Gamma(S^2(T^*M))$ and $t \in \R$, we write $g_t=g+th$ as well as 
\[
\Phi_{g_t}:=\left(\nabla-{}^{g_t}\nabla\right)_0 \qquad \text{and} \qquad \Phi_g:=\Phi_{g_0}.
\]
Let $\zeta$ denote the map which assigns to a Riemannian metric $g$ its Levi-Civita connection and let $\zeta^{\prime}_g : \Gamma(S^2(T^*M)) \to \Gamma(E)$ denote its linearisation at $g \in \mathrm{Riem}(M)$
\[
\zeta^{\prime}_g(h)=\lim_{t\to 0}\frac{1}{t}\left({}^{g_t}\nabla-\levi\right). 
\]
Using $\zeta^{\prime}_g$ we obtain a linear first order differential operator
\[
	\ell_g : \Gamma(S^2(T^*M)) \to \Gamma(E_0), \quad h \mapsto \left(\zeta^{\prime}_g(h)\right)_0.
\]
We let
\[
	\ell^*_g : \Gamma(E_0) \to \Gamma(S^2(T^*M))
\]
denote the formal adjoint satisfying
\begin{equation}\label{eqn:formaladjoint}
	\int_M\langle \Psi,\ell_g(h)\rangle_g d\mu_g=\int_M\langle \ell^*_g(\Psi),h\rangle_g d\mu_g
\end{equation}
for all $h \in \Gamma(S^2(T^*M))$ and all $\Psi \in \Gamma(E_0)$. 

Furthermore, for a section $\Psi$ of $E_0$ we let $\mathsf{q}_g(\Psi)$ denote the unique symmetric covariant $2$-tensor field so that
\[
	\pvar |\Psi|^2_{g_t}d\mu_{g_t}=\langle \mathsf{q}_g(\Psi),h\rangle_g d\mu_g
\]
for all $h \in \Gamma(S^2(T^*M))$. Using the chain rule we compute
\begin{equation}\label{eqn:varchain}
\pvar |\Phi_{g_t}|^2_{g_t}d\mu_{g_t}=\pvar |\Phi_g|^2_{g_t}d\mu_{g_t}+\pvar |\Phi_{g_t}|^2_g d\mu_g.
\end{equation}
By the definition of $\mathsf{q}_g$, the first term equals $\langle \mathsf{q}_g(\Phi_g),h\rangle_g d\mu_g$ and for the second term we obtain $2\langle \Phi_g,\pvar\Phi_{g_t}\rangle_g d\mu_g$. Consequently, we have:
\begin{equation}\label{eqn:var}
\pvar |\Phi_{g_t}|^2_{g_t}d\mu_{g_t}=\langle \mathsf{q}_g(\Phi_g),h\rangle_g d\mu_g+2\langle \Phi_g,\pvar\Phi_{g_t}\rangle_g d\mu_g.
\end{equation}
Writing
\[
	\Phi_{g_t}=(\nabla-\levi-({}^{g_t}\nabla-\levi))_0,
\]
we observe that the difference $\nabla-\levi$ is independent of $t$. Setting $H:=\zeta^{\prime}_g(h)$, the definition of $\zeta^{\prime}_g$ gives
\[
	\pvar({}^{g_t}\nabla-\levi)=H,
\]
and since $\ell_g(h)=H_0$ by definition of $\ell_g$, we conclude
\[
\pvar \Phi_{g_t}=-H_0=-\ell_g(h).
\]
Substituting into \eqref{eqn:var} thus yields
\begin{equation}\label{eqn:var2}
\pvar |\Phi_{g_t}|^2_{g_t}d\mu_{g_t}=\langle \mathsf{q}_g(\Phi_g),h\rangle_gd\mu_g-2\langle \Phi_g,\ell_g(h)\rangle_g d\mu_g.
\end{equation}

Now consider
\[
\hat{\mathcal{E}}(g):=\int_M|(\nabla-\levi)_0|^2_gd\mu_g=\int_M|\Phi_g|^2_gd\mu_g=\Vert\Phi_g\Vert^2_g
\]
so that, interchanging $\pvar$ with $\int_M$ and integrating \eqref{eqn:var2} over $M$, the first two equalities below follow; the third uses the defining property \eqref{eqn:formaladjoint} of $\ell^*_g$ to push $\ell_g$ across the inner product:
\[
\begin{aligned}
\hat{\mathcal{E}}^{\prime}_g(h)=\pvar\int_M|\Phi_{g_t}|^2_{g_t}d\mu_{g_t}&=\int_M\langle \mathsf{q}_g(\Phi_g),h\rangle_g-2\langle \Phi_g,\ell_g(h)\rangle_g d\mu_g\\
&=\int_M\langle \mathsf{q}_g(\Phi_g)-2\ell^*_g(\Phi_g),h\rangle_g d\mu_g.
\end{aligned}
\]

Thinking of the volume of $M$ with respect to a metric as a functional $\Vol_{M} : \mathrm{Riem}(M) \to \R$, we have the standard identity
\[
(\Vol_{M})^{\prime}_g(h)=\frac{1}{2}\int_M(\tr_g\! h)d\mu_g=\frac{1}{2}\int_M\langle g,h\rangle_gd\mu_g.
\]
Since $\mathcal{E}(g)=\Vol_{M}(g)^{(2-n)/n}\hat{\mathcal{E}}(g)$, the product rule gives
\[
\mathcal{E}^{\prime}_g(h)=\Vol_{M}(g)^{(2-n)/n}\hat{\mathcal{E}}^{\prime}_g(h)+\tfrac{2-n}{n}\Vol_{M}(g)^{(2-n)/n-1}\hat{\mathcal{E}}(g)\,(\Vol_M)^{\prime}_g(h).
\]
For a smooth real-valued function $u$ on $M$ we write $\overline{u}:=\Vol_M(g)^{-1}\int_Mu\,d\mu_g$ for its mean value. Substituting the formula for $(\Vol_M)^{\prime}_g(h)$ above and using $\hat{\mathcal{E}}(g)=\Vert\Phi_g\Vert^2_g=\Vol_M(g)\,\overline{|\Phi_g|^2_g}$, we obtain
\[
\mathcal{E}^{\prime}_g(h)=\Vol_{M}(g)^{(2-n)/n}\left[\hat{\mathcal{E}}^{\prime}_g(h)-\left(\frac{n-2}{2n}\right)\overline{|\Phi_g|^2_g}\int_M \langle g,h\rangle_g d\mu_g\right].
\]
Combining this with the expression for $\hat{\mathcal{E}}^{\prime}_g(h)$ derived above, we conclude
\begin{equation}\label{eqn:varE}
\mathcal{E}^{\prime}_g(h)=\Vol_{M}(g)^{(2-n)/n}\int_M\Big\langle \mathsf{q}_g(\Phi_g)-2\ell^*_g(\Phi_g)-\left(\frac{n-2}{2n}\right)\overline{|\Phi_g|^2_g}\,g,\,h\Big\rangle_g d\mu_g.
\end{equation}
By the fundamental lemma of the calculus of variations, $\mathcal{E}^{\prime}_g(h)=0$ for all $h \in \Gamma(S^2(T^*M))$ if and only if the symmetric tensor $\mathsf{q}_g(\Phi_g)-2\ell^*_g(\Phi_g)-\left(\frac{n-2}{2n}\right)\overline{|\Phi_g|^2_g}\,g$ vanishes identically. We thus have:
\begin{prop}\label{prop:vareq}
A Riemannian metric $g$ is critical for $\mathcal{E}$ if and only if
\[
\mathsf{q}_g(\Phi_g)-2\ell^*_g(\Phi_g)=\left(\frac{n-2}{2n}\right)\overline{|\Phi_g|^2_g}\,g.
\]
\end{prop}
We also need:
\begin{lem}\label{lem:quadraticterm}
For $\Psi\in \Gamma(E_0)$ we have
\[
	\mathsf{q}_g(\Psi)=\frac{1}{2}|\Psi|^2_gg+\Psi\mal \Psi,
\]
where
\[
(\Psi\mal\Psi)_{ij}=\Psi^k_{uv}\Psi^a_{bc}g_{ia}g_{jk}g^{ub}g^{vc}-2\Psi^k_{jc}\Psi^a_{ib}g_{ak}g^{bc}.
\]
In particular
\begin{equation}\label{eqn:tracequadratic}
	\tr_g\left(\mathsf{q}_g(\Psi)\right)=\left(\tfrac{n}{2}-1\right)|\Psi|^2_g.
\end{equation}
\end{lem}
\begin{proof}
Recall the standard identities
\[
\pvar d\mu_{g_t}=\frac{1}{2}(\tr_g\!h) d\mu_g\qquad \text{and}\qquad \pvar\left(g_t^{\sharp}\right)^{ij}=-h^{ij}
\]
where we raise indices of $h$ with the metric $g$. Then
\begin{multline*}
|\Psi|^2_{g_t}d \mu_{g_t}=\Psi^k_{ij}\Psi^a_{bc}\left(g_{ak}+th_{ak}\right)
\left(g^{ib}-th^{ib}+\mathcal{O}(t^2)\right)\\\left(g^{jc}-th^{jc}+\mathcal{O}(t^2)\right)\left(1+\frac{1}{2}tg_{pq}h^{pq}+\mathcal{O}(t^2)\right)d\mu_{g}.
\end{multline*}
From this we calculate
\begin{multline*}
\pvar |\Psi|^2_{g_t}d\mu_{g_t}=\Big[\Psi^k_{uv}\Psi^a_{bc}g_{ia}g_{jk}g^{ub}g^{vc}-2\Psi^k_{jc}\Psi^a_{ib}g_{ak}g^{bc}\Big.\\\Big.+\tfrac{1}{2}|\Psi|^2_gg_{ij}\Big]h^{ij}d\mu_g,
\end{multline*}
so that
\[
	\pvar |\Psi|^2_{g_t}d\mu_{g_t}=\langle \mathsf{q}_g(\Psi),h\rangle_g d\mu_g,
\]
as claimed. The identity \eqref{eqn:tracequadratic} follows from an elementary computation which we omit. 
\end{proof}
\begin{lem}
For $g \in \mathrm{Riem}(M)$ the operator
\[
	\ell^*_g : \Gamma(E_0) \to \Gamma(S^2(T^*M))
\]
is given by
\[
\Psi^i_{jk} \mapsto \tfrac{1}{2}\left(\levi_k\Psi^k_{ij}-g^{uv}g_{ik}\levi_u\Psi^k_{jv}-g^{uv}g_{jk}\levi_u\Psi^k_{iv}\right).
\]	
\end{lem}
\begin{proof}
We need to show that for $g \in \Riem$, $h \in \Gamma(S^2(T^*M))$ and $\Psi \in \Gamma(E_0)$, we have
\[
	\int_M\langle \Psi,\ell_g(h)\rangle_g d\mu_g=\int_M\langle \ell^*_g(\Psi),h\rangle_g d\mu_g.
\]
With our notation $H=\zeta^{\prime}_g(h)$, an elementary computation gives
\[
H^k_{ij}=\frac{1}{2}g^{kl}\left(\levi_ih_{jl}+\levi_j h_{il}-\levi_lh_{ij}\right).
\]
Recall from \eqref{eqn:glndecomp} that
\[
	H=H_0+\tfrac{1}{(n+1)}\sym(\con(H)).
\]
Since $\Psi \in \Gamma(E_0)$ and $\ell_g(h)=H_0$, we have 
\[
\langle \Psi,H\rangle_g=\langle \Psi,H_0\rangle_g+\tfrac{1}{(n+1)}\langle \Psi,\sym(\con(H))\rangle_g=\langle \Psi,H_0\rangle_g=\langle \Psi,\ell_g(h)\rangle_g,
\]
where we use \eqref{eqn:sondecomp}. 
Contracting $g_{ka}g^{al}=\delta^l_k$ gives
\[
\langle \Psi,H\rangle_g=\frac{1}{2}\Psi^k_{ij}g^{ib}g^{jc}\left(\levi_b h_{ck}+\levi_c h_{bk}-\levi_k h_{bc}\right).
\]
The first two terms contribute equally, so that
\[
\langle \Psi,H\rangle_g=\Psi^k_{ij}g^{ib}g^{jc}\,\levi_b h_{ck}-\frac{1}{2}\Psi^k_{ij}g^{ib}g^{jc}\,\levi_k h_{bc}.
\]
Consider the vector field $Y=(Y^{l})$ on $M$ with 
\[
	Y^{l}=g^{il}g^{jc}\Psi^k_{ij}h_{ck}.
\]
Since $\int_M \div_g\!Y\,d\mu_g=0$ by Stokes' theorem, we obtain
\[
	\int_M g^{ib}g^{jc}\,\levi_b(\Psi^k_{ij}h_{ck})d \mu_g=0 
\]
and hence
\[
	\int_Mg^{ib}g^{jc}\Psi^k_{ij}\levi_b h_{ck}\,d\mu_g=-\int_M g^{ib}g^{jc}(\levi_b\Psi^k_{ij})\,h_{ck}\,d\mu_g.
\]
Likewise we obtain
\[
-\frac{1}{2}\int_M\Psi^k_{ij}g^{ib}g^{jc}\,\levi_k h_{bc}\,d\mu_g=\frac{1}{2}\int_M g^{ib}g^{jc}(\levi_k\Psi^k_{ij})\,h_{bc}\,d\mu_g.
\]
In the second integral, substituting $h_{bc}=g_{bp}g_{cq}h^{pq}$ gives $\frac{1}{2}(\levi_k\Psi^k_{pq})\,h^{pq}$. In the first, substituting $h_{ck}=g_{cp}g_{kq}h^{pq}$ gives $-g^{ub}g_{kq}(\levi_b\Psi^k_{up})\,h^{pq}$. Symmetrising in $p,q$ and relabelling summation indices (using $\Psi^k_{up}=\Psi^k_{pu}$) yields
\[
[\ell^*_g(\Psi)]_{pq}=\frac{1}{2}\left(\levi_k\Psi^k_{pq}-g^{uv}g_{pk}\levi_u\Psi^k_{qv}-g^{uv}g_{qk}\levi_u\Psi^k_{pv}\right). \qedhere
\]
\end{proof}

\subsection{Decomposing $\bm{\Phi_g}$} 

Decomposing $\Phi_g$ as in \eqref{eqn:decomp}
\[
	X_g:=\frac{1}{m}\tr_g \Phi_g\qquad \text{and} \qquad A_{g}:=\Phi_g-(g\otimes X_g)_0
\]
it turns out that $A_{g}$ depends only on the conformal equivalence class $[g]$ of $g$. To this end let $f \in C^{\infty}$ and consider $\tg=\exp(2f)g$. Using the identity (see \cite[Thm.~1.159]{MR867684})
\begin{equation}\label{eqn:besse}
{}^{\tg}\nabla=\levi-g\otimes\grad_g\! f+\sym(\d f),
\end{equation}
we obtain
\[
\Phi_{\tg}=(\nabla-{}^{\tg}\nabla)_0=\Phi_g+(g\otimes \grad_g f)_0=\Phi_g+g\otimes \grad_g f-\frac{1}{(n+1)}\sym(\d f),
\]
where we use that $\con(g\otimes \grad_g f)=\d f$.
Hence we compute
\begin{equation}\label{eq:confX}
\begin{aligned}
X_{\tg}&=\frac{1}{m}\tr_{\tg}\!\Phi_{\tg}=\e^{-2f}\left(X_g+\frac{n}{m}\grad_g f-\frac{2}{(n+1)m}\grad_g f\right)\\
&=\e^{-2f}\left(X_g+\grad_g f\right).
\end{aligned}
\end{equation}
This also gives
\[
\begin{aligned}
	A_{\tg}&=\Phi_{\tg}-(\tg\otimes X_{\tg})_0\\
	&=\Phi_g+g\otimes \grad_g f-\tfrac{1}{(n+1)}\sym(\d f)-\left(g\otimes(X_g+\grad_g f)\right)_0\\
	&=A_g+g\otimes \grad_g f-\tfrac{1}{(n+1)}\sym(\d f)-(g\otimes \grad_g f)_0\\
	&=A_g-\tfrac{1}{(n+1)}\sym(\d f)+\tfrac{1}{(n+1)}\sym(\con(g\otimes \grad_g f))=A_{g}. 
\end{aligned}
\]
We will henceforth write $\higgs$ instead of $A_g$. The conformal invariance of $\higgs$ was previously observed in \cite[Thm.~2.4]{MR4201191}. We included the argument for the convenience of the reader. 

For what follows we recall the standard identity for a vector field $Y$ on $M$
\[
\div_{\tg}Y=\div_g\! Y+nY(f),
\]
which, together with \eqref{eq:confX}, gives
\begin{equation}\label{eq:confchangediv}
\div_{\tg}X_{\tg}=\e^{-2f}\left(\div_g\! X_g+(n-2)X_g(f)-\Delta_g f+(n-2)|\d f|^2_g\right),
\end{equation}
where $\Delta_g=-\div_g \grad_g$.

\subsection{Conformal variations}

As an immediate corollary of \cref{prop:vareq} we have:
\begin{cor}\label{cor:confvar}
A metric $g$ is conformally critical for $\mathcal{E}$ if and only if
\begin{equation}\label{eq:confcrit}
\div_g(\tr_g\!\Phi_g)=\left(\frac{n}{2}-1\right)\left(|\Phi_g|^2_g-\overline{|\Phi_g|^2_g}\right).
\end{equation}
\end{cor}
\begin{proof}
Observe that we have
\[
\tr_g\left(\ell^*_g(\Phi_g)\right)=\tfrac{1}{2}\div_g(\tr_g\Phi_g),
\]
where
\[
	\div_g(\tr_g\Phi_g)=\levi_ig^{jk}(\Phi_g)^i_{jk}.
\]
Furthermore, recall \eqref{eqn:tracequadratic}
\[
	\tr_g(\mathsf{q}_g(\Phi_g))=\left(\tfrac{n}{2}-1\right)|\Phi_g|^2_g.
\]
For $h=f g$ with $f \in C^{\infty}$ we obtain
\[
\begin{aligned}
\Vol_{M}(g)^{(n-2)/n}\,\mathcal{E}^{\prime}_g(fg)&=\int_M f\left[\tr_g\left(\mathsf{q}_g(\Phi_g)-2\ell^*_g(\Phi_g)\right)-\big(\tfrac{n}{2}-1\big)\overline{|\Phi_g|^2_g}\right]d\mu_g\\
&=\int_M f\left[\big(\tfrac{n}{2}-1\big)\left(|\Phi_g|^2_g-\overline{|\Phi_g|^2_g}\right)-\div_g(\tr_g\Phi_g)\right]d\mu_g,
\end{aligned}
\]
and the claim follows.
\end{proof}

\section{The surface case}

\subsection{Example: Projective structures from holomorphic cubic differentials}
Let $(M,g)$ be an oriented Riemannian $2$-manifold and let $J : TM \to TM$ denote anticlockwise rotation by $\pi/2$ with respect to the orientation and metric $g$. We consider a cubic differential $C$ on the Riemann surface $(M,J)$. Its real part is a totally symmetric $(0,3)$ tensor field $A=(A_{ijk})$ on $M$ that is totally trace-free with respect to $g$, that is, $g^{ij}A_{ijk}$ vanishes identically. Define $\alpha=(\alpha^i_{jk}) \in \Gamma(E)$ by $\alpha^i_{jk}=g^{il}A_{ljk}$ and consider the projective structure $\mathfrak{p}$ arising from $\levi+\alpha$. By construction, $\alpha \in \Gamma(E_0)$ and $\Phi_g=\alpha$. Applying \cref{prop:vareq}, the variational equations become $\mathsf{q}_g(\alpha)=2\ell^*_g(\alpha)$. By \eqref{eqn:tracequadratic}, the tensor field $\mathsf{q}_g(\alpha)$ is totally trace-free with respect to $g$ and hence encodes a quadratic differential on $(M,J)$. The map $\mathsf{q}_g$ thus gives rise to a quadratic form on $\Gamma(K^3_{M,J})$ with values in $\Gamma(K^2_{M,J})$, where $K_{M,J}$ denotes the canonical bundle of $(M,J)$. Unsurprisingly, this quadratic form is trivial (as can be checked with a direct computation), hence $\mathsf{q}_g(\alpha)$ vanishes identically and the variational equations simplify to $\ell^*_g(\alpha)=0$. A computation gives $2\ell^*_g(\alpha)=-\div_g \alpha$ with $\left(\div_g \alpha\right)_{ij}=\levi_k \alpha^k_{ij}$. This last condition is well known to be equivalent to the holomorphicity of $C$. From this we obtain:

\maintheoremone*

\begin{proof}
By \cite{MR1828223,MR2402597}, on a closed oriented surface $M$ of negative Euler characteristic every \emph{properly convex projective structure} $\mathfrak{p}$ arises from a unique pair $(g,C)$ with $C$ holomorphic via the above construction. That is, the torsion-free connection $\levi+\alpha$ on $TM$ is a representative connection of $\proj$. Additionally, the pair satisfies $K_g=-1+2|C|^2_g$, where $K_g$ denotes the Gauss curvature of $g$. The metric $g$ is known as the \emph{Blaschke metric} of $\mathfrak{p}$. The  above computations then show that $\mathsf{q}_g(\alpha)=0$ and $\ell^*_g(\alpha)=0$ by the holomorphicity of $C$, so the Blaschke metric $g$ is critical for $\mathcal{E}$.
\end{proof}

\subsection{Existence and uniqueness of conformally critical metrics}

\cref{cor:confvar} implies that for $n=2$ a Riemannian metric $g$ is conformally critical if and only if $\div_g\!X_g=0$. Using \eqref{eq:confchangediv}, looking for a conformally critical metric $\tg=\e^{2f}g$ in the conformal equivalence class of $g$ thus gives the PDE
\begin{align*}
\div_{\tg}X_{\tg}=\e^{-2f}\left(\div_g\! X_g-\Delta_g f\right)=0.
\end{align*}
Hence we obtain:
\maintheoremtwo*

\begin{proof}
Recall that for $u \in C^{\infty}$ the Poisson equation $\Delta_gf=u$ on a closed oriented Riemannian manifold $(M,g)$ has a smooth solution, unique up to adding a constant, provided $\int_Mud\mu_g=0$. In our case $u=\div_g\! X_g$, hence the claim is a consequence of Stokes' theorem.  
\end{proof}

\section{The projective Yamabe equation} 

From now on we assume $n\geqslant 3$. Consider a closed oriented conformal $n$-manifold $(M,[g])$. Recall that the Yamabe problem asks whether there exists a metric $\tilde{g} \in [g]$ whose scalar curvature $S_{\tg}$ is constant. For $f \in C^{\infty}$ the scalar curvature of the metric $\tilde{g}=\exp(2f)g$ is
\[
S_{\tilde{g}}=\e^{-2f}\left(S_g+2(n-1)\Delta_g f-(n-1)(n-2)|\d f|^2_g\right),
\]
where $\Delta_g f=-\div_g\grad_g f$. Writing $\tg=u^{4/(n-2)}g$ for a positive function $u$, the previous equation becomes
\[
4\left(\frac{n-1}{n-2}\right)\Delta_gu+S_gu=u^{(n+2)/(n-2)}S_{\tg}.
\]
Finding a metric $\tilde{g} \in [g]$ with constant scalar curvature $\kappa$ thus amounts to solving $L_g u=\kappa u^{(n+2)/(n-2)}$, where $L_g=4\left(\tfrac{n-1}{n-2}\right)\Delta_g+S_g$ denotes the so-called \emph{conformal Laplacian}.

\subsection{The projective-conformal Laplacian}\label{sub:the_projective_conformal_laplacian}

We now return to the problem of finding a conformally critical metric on a closed oriented projective manifold of dimension $n\geqslant 3$. Setting
\[\hat{m}:=\frac{(n+1)(n-2)}{(n+2)},
\] so that $\hat{m}m=(n-1)(n-2)$, we define
\begin{equation}\label{eq:defformscalcurv}
\mathcal{S}_g:=\hat{m}|\Phi_g|^2_g-2(n-1)\div_g X_g.
\end{equation}
Since $\tr_g\Phi_g=mX_g$, we have
\[\mathcal{S}_g=\tfrac{2(n+1)}{(n+2)}\big[\left(\tfrac{n}{2}-1\right)|\Phi_g|^2_g-\div_g(\tr_g\Phi_g)\big],
\] and hence, by \cref{cor:confvar}, a metric $g$ is conformally critical for $\mathcal{E}$ if and only if $\mathcal{S}_g$ is constant.  
\begin{lem}\label{lem:projectivescalar}
Let $f \in C^{\infty}$. Under a conformal change $g \mapsto \tg=\e^{2f}g$, the function $\mathcal{S}_g$ transforms as the scalar curvature, that is,
\[
\mathcal{S}_{\tg}=\e^{-2f}(\mathcal{S}_g+2(n-1)\Delta_g f-(n-1)(n-2)|\d f|^2_g),
\]
where $\Delta_g f=-\div_g\!\grad_g\! f$. 
\end{lem} 
\begin{proof} 
From the formula 
\[
	\Phi_{\tg}=\Phi_g+g\otimes \grad_g f-\frac{1}{(n+1)}\sym(\d f)
\]
we compute
\begin{equation}\label{eq:confchangehiggs}
\begin{aligned}
|\Phi_{\tg}|^2_{\tg}&=\e^{-2f}|\Phi_g+(g\otimes\grad_g\! f)_0|^2_g\\
&=\e^{-2f}\left(|\Phi_g|^2_g+ 2m X_g(f)+|(g\otimes \grad_g\! f)_0|^2_g\right)\\
&=\e^{-2f}\left(|\Phi_g|^2_g+2mX_g(f)+m|\d f|^2_g\right),
\end{aligned}
\end{equation}
where we have used the identities
\[
|\sym(\d f)|^2_g=2(n+1)|\d f|^2_g\quad \text{and}\quad |g\otimes \grad_g\! f|^2_g=n|\d f|^2_g
\]
as well as
\[
\langle g\otimes \grad_g\! f,\sym (\d f)\rangle_g=2|\d f|^2_g.
\]
Now using \eqref{eq:confchangediv} and \eqref{eq:confchangehiggs} we obtain
\begin{align*}
\mathcal{S}_{\tg}&=\hat{m}|\Phi_{\tg}|^2_{\tg}-2(n-1)\div_{\tg}X_{\tg}\\
&=\e^{-2f}\left[\mathcal{S}_g+2(n-1)(n-2)X_g(f)+(n-1)(n-2)|\d f|^2_g\right.\\
&\phantom{=}\;\left.-2(n-1)(n-2)X_g(f)+2(n-1)\Delta_g f-2(n-1)(n-2)|\d f|^2_g\right]\\
&=\e^{-2f}\left(\mathcal{S}_g+2(n-1)\Delta_g f-(n-1)(n-2)|\d f|^2_g\right),
\end{align*}
as claimed, where we used $\hat{m}m=(n-1)(n-2)$.
\end{proof}
Finding a metric $\tilde{g}$ in a conformal equivalence class $[g]$ which is conformally critical for $\mathcal{E}$ is thus equivalent to solving an equation of Yamabe type $\mathcal{L}_g u=\kappa u^{(n+2)/(n-2)}$,
where 
\[
	\mathcal{L}_g:=4\left(\frac{n-1}{n-2}\right)\Delta_g + \mathcal{S}_g.
\]
\begin{rmk}
We will refer to $\mathcal{L}_g$ as the \emph{projective-conformal Laplacian}.
\end{rmk}

\subsection{The projective scalar curvature}

As an immediate consequence of Stokes' theorem and definition \eqref{eq:defformscalcurv} we have
\begin{prop}
Let $g$ be a Riemannian metric on the closed oriented projective manifold $(M,\proj)$ of dimension $n \geqslant 3$. Then 
\begin{equation}\label{eq:integralformula}
  \int_M\mathcal{S}_g\,d\mu_g
  = \hat{m} \,\Vert\Phi_g\Vert^2_g
  \geqslant 0,
\end{equation}
with equality if and only if $\mathfrak{p}$ contains the Levi-Civita connection of $g$. 
\end{prop}

Denoting by $\sigma(\mathcal{L}_g)$ the spectrum of $\mathcal{L}_g$, we also conclude:
\begin{prop}\label{prop:spectrum}
Let $(M,\proj)$ be a closed oriented projective manifold of dimension $n\geqslant 3$ and $g$ a Riemannian metric on $M$. Then $\sigma(\mathcal{L}_g)$ is non-negative and moreover, contains $0$ if and only if $\proj$ is metrisable by a metric $\tilde{g}$ in the conformal equivalence class of $g$.
\end{prop}
\begin{proof}
For a positive function $u$ on $M$, write $\tilde{g}=u^{4/(n-2)}g$ and use $\mathcal{S}_{\tilde{g}}=u^{-(n+2)/(n-2)}\mathcal{L}_gu$ as well as $d\mu_{\tilde{g}}=u^{2n/(n-2)}d\mu_g$, then we obtain
\[
	\int_Mu\,\mathcal{L}_gu\,d\mu_g=\int_Mu^{2n/(n-2)}\mathcal{S}_{\tilde{g}}\,d\mu_g=\int_M\mathcal{S}_{\tilde{g}}\,d\mu_{\tilde{g}}=\hat{m}\,\Vert\Phi_{\tilde{g}}\Vert^2_{\tilde{g}}\geqslant 0.
\]
Since, by standard elliptic PDE theory, the first eigenfunction of the self-adjoint elliptic operator $\mathcal{L}_g$ can be chosen to be positive, it follows that the first eigenvalue $\lambda_1$ of $\mathcal{L}_g$ satisfies $\lambda_1\geqslant 0$ and hence the full spectrum is non-negative. If $\lambda_1=0$, the corresponding positive eigenfunction $u$ satisfies $\mathcal{L}_gu=0$, so that $\Vert\Phi_{\tilde{g}}\Vert^2_{\tilde{g}}=0$ for $\tilde{g}=u^{4/(n-2)}g$, meaning $\proj$ contains the Levi-Civita connection of $\tilde{g}\in [g]$. On the other hand, if $\tilde{g}=u^{4/(n-2)}g$ metrises $\proj$, then $\Phi_{\tilde{g}}=0$ and hence $\mathcal{L}_gu=u^{(n+2)/(n-2)}\mathcal{S}_{\tilde{g}}=0$. It follows that $0$ is in the spectrum of $\mathcal{L}_g$.
\end{proof}

Because of the analogy to the scalar curvature, we will refer to $\mathcal{S}_g$ as the \emph{projective scalar curvature} of $g$. For what follows it is advantageous to express $\mathcal{S}_g$ in terms of the usual scalar curvature of $g$, the ``scalar curvature'' of a certain Weyl connection and $\higgs$. It follows from \eqref{eqn:besse} and \eqref{eq:confX} that the torsion-free connection 
\begin{equation}\label{eqn:weylcon}
	\weyl:=\levi+g\otimes X_g-\sym(X_g^{\flat})
\end{equation}
depends only on the conformal equivalence class of $g$. Here $\flat$ is taken with respect to $g$. The connection $\weyl$ on $TM$ is a \emph{Weyl connection} for $[g]$, that is, it is torsion-free and its parallel transport maps are angle preserving with respect to $[g]$. In fact, all Weyl connections for $[g]$ are of the form \eqref{eqn:weylcon} for some metric $g \in [g]$ and some vector field $X$. For a pair $(g,X)$ we write
\[
	S_{g,X}:=\tr_g \mathrm{Ric}\left(\weyl\right),
\]
where $\mathrm{Ric}\left(\weyl\right)$ denotes the Ricci curvature of the Weyl connection $\weyl$ determined by $(g,X)$. A standard calculation -- see \cref{sec:ricci_curvature_of_a_weyl_connection} -- gives:
\begin{equation}\label{eq:weylscal}
  S_{g,X}
  =S_g+2(n-1)\div_g\!X-(n-2)(n-1)|X|^2_g,
\end{equation}
where $S_g$ denotes the scalar curvature of $g$.
\begin{prop}
For the projective scalar curvature we have
\[
	\mathcal{S}_g=S_g-S_{g,X_g}+\hat{m} |\higgs|^2_g.
\]
\end{prop}
\begin{proof}
Using the definition \eqref{eq:defformscalcurv}, \eqref{eq:weylscal} and \eqref{eqn:normsplit}, we compute
\[
\begin{aligned}
\mathcal{S}_g&=\hat{m} |\Phi_g|^2_g-2(n-1)\div_g(X_g)\\
&=\hat{m} |\higgs|^2_g+\hat{m} m |X_g|^2_g-2(n-1)\div_g(X_g)\\
&=S_g-\big(S_g+2(n-1)\div_g(X_g)-(n-2)(n-1)|X_g|^2_g\big)+\hat{m} |\higgs|^2_g\\
&=S_g-S_{g,X_g}+\hat{m} |\higgs|^2_g,
\end{aligned}
\]
where we also used that $\hat{m}m=(n-1)(n-2)$.
\end{proof}
\begin{rmk}
In terms of $X_g$ and $\higgs$ we thus have
\[
	\mathcal{S}_g=\hat{m} |\higgs|^2_g+(n-2)(n-1)|X_g|^2_g-2(n-1)\div_g(X_g).
\]
\end{rmk}

\subsection{The projective Einstein tensor}\label{sub:projective_einstein}

Recall that the first variation of the total scalar curvature is 
\[\pvar\int_M S_{g_t}\,d\mu_{g_t}=-\int_M\langle G_g,h\rangle_g\,d\mu_g,
\]
where $G_g=\Ric_g-\tfrac12S_gg$ is the Einstein tensor, and that a metric $g$ is critical for the (volume-normalised) Einstein--Hilbert functional if and only if $G_g+\Lambda g=0$ with $\Lambda=\left(\tfrac{n-2}{2n}\right)\overline{S_g}$. Since $\mathcal{E}$ is, up to the constant factor $\hat{m}^{-1}$, the Einstein--Hilbert functional with $S_g$ replaced by $\mathcal{S}_g$, we define the \emph{projective Einstein tensor} $\mathcal{G}_g\in\Gamma(S^2(T^*M))$ of a metric $g$ on $(M,\proj)$ by the same characterisation:
\begin{equation}\label{eq:defprojeinstein}
\pvar\int_M\mathcal{S}_{g_t}\,d\mu_{g_t}=-\int_M\langle\mathcal{G}_g,h\rangle_g\,d\mu_g\qquad \text{for all } h\in\Gamma(S^2(T^*M)).
\end{equation}
For a symmetric $2$-tensor field $\tau$ we denote by $\div_g\tau$ the $1$-form with components $(\div_g\tau)_j=g^{ik}\,\levi_k\tau_{ij}$. Furthermore, we denote by $\mathrm{Proj}(M)$ the set of projective structures on $M$. By \cref{lem:projequiv}, $\mathrm{Proj}(M)$ is an affine space modelled on $\Gamma(E_0)$, the difference of $\proj,\proj^{\prime}\in\mathrm{Proj}(M)$ being $(\nabla^{\prime}-\nabla)_0$ for any $\nabla\in\proj$ and $\nabla^{\prime}\in\proj^{\prime}$. Let
\[
\mathcal{F} : \mathrm{Riem}(M)\times\mathrm{Proj}(M)\to\R,\qquad (g,\proj)\mapsto\int_M|(\nabla-\levi)_0|^2_g\,d\mu_g,\quad \nabla\in\proj,
\]
so that $\hat{\mathcal{E}}(g)=\mathcal{F}(g,\proj)$. An orientation preserving diffeomorphism $\varphi$ of $M$ acts on both factors by pull-back, $\varphi^*\proj$ being the projective structure containing $\varphi^*\nabla$ for $\nabla\in\proj$. For a vector field $Z$ on $M$ with flow $\varphi_t$ we write
\[
\mathrm{L}_Z\proj:=\pvar\varphi_t^*\proj\in\Gamma(E_0).
\]
Equivalently, $\mathrm{L}_Z\proj=(\mathrm{L}_Z\nabla)_0$ for $\nabla\in\proj$, where $\mathrm{L}_Z\nabla\in\Gamma(E)$ denotes the Lie derivative of $\nabla$. The vector field $Z$ is called a \emph{projective vector field} if $\mathrm{L}_Z\proj=0$ (see \cite[Ch.~I, \S2 and Ch.~III, \S1]{MR88769}).

\begin{prop}\label{prop:projeinstein}
Let $(M,\proj)$ be a closed oriented projective manifold of dimension $n\geqslant 3$ and $g$ a Riemannian metric on $M$. Then:
\begin{enumerate}
	\item[(i)] the projective Einstein tensor is given by
	\[
		\mathcal{G}_g=\hat{m}\left(2\ell^*_g(\Phi_g)-\mathsf{q}_g(\Phi_g)\right);
	\]
	\item[(ii)] the metric $g$ is critical for $\mathcal{E}$ if and only if
	\[
		\mathcal{G}_g+\Lambda g=0,\qquad\text{where}\qquad \Lambda=\left(\frac{n-2}{2n}\right)\overline{\mathcal{S}_g};
	\]
	\item[(iii)] as for the usual Einstein tensor, we have $\tr_g\mathcal{G}_g=\left(1-\tfrac{n}{2}\right)\mathcal{S}_g$;
	\item[(iv)] for every vector field $Z$ on $M$
	\[
		\int_M(\div_g\mathcal{G}_g)(Z)\,d\mu_g=-\hat{m}\int_M\langle\Phi_g,\mathrm{L}_Z\proj\rangle_g\,d\mu_g.
	\]
\end{enumerate}
\end{prop}
\begin{proof}
(i) By \eqref{eq:integralformula}, applied to each metric $g_t$, we have $\int_M\mathcal{S}_{g_t}\,d\mu_{g_t}=\hat{m}\,\hat{\mathcal{E}}(g_t)$. The expression for $\hat{\mathcal{E}}^{\prime}_g(h)$ obtained in the proof of \cref{prop:vareq} thus gives
\[
\pvar\int_M\mathcal{S}_{g_t}\,d\mu_{g_t}=-\hat{m}\int_M\langle 2\ell^*_g(\Phi_g)-\mathsf{q}_g(\Phi_g),h\rangle_g\,d\mu_g,
\]
and the claim follows from \eqref{eq:defprojeinstein}.

(ii) By \eqref{eq:integralformula} we have $\Lambda=\hat{m}\left(\frac{n-2}{2n}\right)\overline{|\Phi_g|^2_g}$. Hence, using (i), $\mathcal{G}_g+\Lambda g=-\hat{m}\big(\mathsf{q}_g(\Phi_g)-2\ell^*_g(\Phi_g)-\left(\frac{n-2}{2n}\right)\overline{|\Phi_g|^2_g}\,g\big)$, and since $\hat{m}\neq 0$ for $n\geqslant 3$, the claim follows from \cref{prop:vareq}.

(iii) Recall from the proof of \cref{cor:confvar} that $\tr_g\ell^*_g(\Phi_g)=\tfrac12\div_g(\tr_g\Phi_g)$. Together with \eqref{eqn:tracequadratic} and (i) this gives
\[
\tr_g\mathcal{G}_g=-\hat{m}\left[\left(\frac{n}{2}-1\right)|\Phi_g|^2_g-m\div_gX_g\right]=\left(1-\frac{n}{2}\right)\mathcal{S}_g,
\]
where the last equality follows from \eqref{eq:defformscalcurv} and $\hat{m}m=(n-1)(n-2)$.

(iv) Since the Levi-Civita connection is natural, we have 
\[(\varphi^*\nabla-{}^{\varphi^*g}\nabla)_0=\varphi^*(\nabla-\levi)_0\]
for every orientation preserving diffeomorphism $\varphi$ of $M$, and hence
\[
\mathcal{F}(\varphi^*g,\varphi^*\proj)=\mathcal{F}(g,\proj).
\]
Now observe that for the derivative we obtain
\[
	\mathcal{F}^{\prime}_{(g,\proj)}((h,\Psi))=\int_M\langle\mathsf{q}_g(\Phi_g)-2\ell^*_g(\Phi_g),h\rangle_g\,d\mu_g+2\int_M\langle \Phi_g,\Psi\rangle_g d\mu_g
\]
for all $h \in \Gamma(S^2(T^*M))$ and all $\Psi \in \Gamma(E_0)$. Differentiating $\mathcal{F}(\varphi_t^*g,\varphi_t^*\proj)=\mathcal{F}(g,\proj)$ at $t=0$ gives
\[
0=\int_M\langle \mathsf{q}_g(\Phi_g)-2\ell^*_g(\Phi_g),\mathrm{L}_Zg\rangle_g\,d\mu_g+2\int_M\langle\Phi_g,\mathrm{L}_Z\proj\rangle_g\,d\mu_g.
\]
For a symmetric $2$-tensor field $\tau$, integrating by parts gives $\int_M\langle\tau,\mathrm{L}_Zg\rangle_g\,d\mu_g=-2\int_M(\div_g\tau)(Z)\,d\mu_g$. Therefore
\[
\int_M\big(\div_g\left(\mathsf{q}_g(\Phi_g)-2\ell^*_g(\Phi_g)\right)\big)(Z)\,d\mu_g=\int_M\langle\Phi_g,\mathrm{L}_Z\proj\rangle_g\,d\mu_g,
\]
and the claim follows from (i).
\end{proof}

\begin{rmk}\leavevmode
\begin{enumerate}
	\item[(i)] For the usual Einstein--Hilbert functional, varying the metric $g$ in $\tr_g(\Ric_g)\,d\mu_g$ with $\Ric_g$ held fixed gives $-\langle G_g,h\rangle_g\,d\mu_g$, and the variation of the Ricci tensor $\Ric_g$ does not contribute to the integral by the Palatini identity. For $\mathcal{E}$, the analogous term $2\hat{m}\,\ell^*_g(\Phi_g)$ coming from the variation of $\Phi_g$ does however contribute.
	\item[(ii)] The Einstein tensor is divergence-free because the Einstein--Hilbert functional is invariant under all diffeomorphisms of $M$, whereas $\mathcal{E}$ is invariant only under those preserving $\proj$. Accordingly, $\div_g\mathcal{G}_g$ does not vanish in general, but $\int_M(\div_g\mathcal{G}_g)(Z)\,d\mu_g=0$ for every projective vector field $Z$. If $g$ is critical for $\mathcal{E}$, then $\div_g\mathcal{G}_g=-\div_g(\Lambda g)=0$, and hence $\int_M\langle\Phi_g,\mathrm{L}_Z\proj\rangle_g\,d\mu_g=0$ for all vector fields $Z$ on $M$.
	\item[(iii)] The Einstein-type formulation degenerates in dimension two. Dropping the assumption $n\geqslant 3$ for a moment, for $n=2$ we have $\hat{m}=0$, so that $\int_M\mathcal{S}_g\,d\mu_g=0$ for every metric $g$ and hence $\mathcal{G}_g$ vanishes identically, although $\mathcal{E}$ is not constant. This parallels the fact that the Einstein tensor of a surface vanishes identically. By contrast, the formulation of \cref{prop:vareq} in terms of $\mathsf{q}_g$ and $\ell^*_g$ holds in every dimension $n\geqslant 2$, and it is the one used in the surface case in \cref{thm:maintheoremone} and \cref{thm:maintheoremtwo}.
\end{enumerate}
\end{rmk}

\section{Existence of conformally critical metrics}

For a Riemannian metric $g$ on the projective manifold $(M,\proj)$ we define:
\begin{defn}
The \emph{conformal defect} of $g$ on $(M,\proj)$ is the difference of the projective and the standard scalar curvature:
\[
	V_g:=\mathcal{S}_g-S_g=\hat{m}|\higgs|^2_g -S_{g,X_g}.
\]
\end{defn}
\begin{rmk}\leavevmode
\begin{enumerate}
	\item[(i)] Notice that $\mathcal{L}_g=L_g+V_g$, so that $V_g$ measures the failure of the projective-conformal Laplacian to agree with the conformal Laplacian.
	\item[(ii)] Recall that we define the conformal locus of $g$ as the set
	\[\loc=\left\{p \in M\,|\, V_g(p)=0\right\}.
	\] Notice that under a conformal change $g \mapsto \tilde{g}=\e^{2f}g$, the conformal defect satisfies $V_{\tilde{g}}=\e^{-2f}V_g$. In particular, the conformal locus depends only on the conformal equivalence class of $g$. 
\end{enumerate}
\end{rmk}

\subsection{Reduction to the Yamabe problem}

The projective-conformal -- and conformal Laplacian agree in the special case where the conformal locus of $[g]$ is all of $M$. Therefore, in this case, finding a conformally critical metric for $\mathcal{E}$ amounts to solving the Yamabe problem. The affirmative solution to the Yamabe problem \cite{MR125546} by Trudinger \cite{MR240748}, Aubin \cite{MR0433500}, and Schoen \cite{MR788292} thus implies (we also refer to \cite{MR888880} and \cite{MR2148873} for surveys):
\begin{thm}\label{thm:fulllocus}
Let $(M,\proj)$ be a closed oriented projective manifold of dimension $n\geqslant 3$. If the conformal locus of conformal structure $[g]$ on $M$ satisfies $\loc=M$, then $[g]$ contains a conformally critical metric for $\mathcal{E}$. 
\end{thm}
\begin{proof}
By assumption the conformal defect vanishes identically so that $\mathcal{L}_g=L_g$ and the criticality equation reduces to the Yamabe equation. Therefore, it has a positive smooth solution by the work of Yamabe, Trudinger, Aubin and Schoen.
\end{proof}

\subsection{Negative conformal defect}\label{ss:negdefect}

We treat in this subsection the case where the conformal defect $V_g$ takes a negative value at some point of $M$. Since $V_{\tg}=\e^{-2f}V_g$ for $\tg=\e^{2f}g$, the hypothesis that $V_g$ is negative at a given point is independent of the choice of representative metric in $[g]$.

\begin{thm}\label{thm:negdefect}
Let $(M,\proj)$ be a closed oriented projective manifold of dimension $n\geqslant 4$ and $[g]$ a conformal structure on $M$. Suppose there exists $p_0\in M$ such that $V_{\tilde{g}}(p_0)<0$ for some (and hence any) $\tilde{g} \in [g]$. Then $[g]$ contains a conformally critical metric for $\mathcal{E}$.
\end{thm}

\begin{proof}
By \cref{sub:the_projective_conformal_laplacian}, finding a metric $\tg\in[g]$ which is conformally critical for $\mathcal{E}$ is equivalent to finding a smooth positive function $u$ on $M$ that solves $\mathcal{L}_gu=\kappa u^{(n+2)/(n-2)}$ for some constant $\kappa\geqslant 0$. Setting $a:=\tfrac{n-2}{4(n-1)}\mathcal{S}_g$ and $\lambda:=\tfrac{n-2}{4(n-1)}\kappa$, this rewrites as the elliptic equation of Yamabe type
\[
\Delta_gu+a\,u=\lambda\,u^{(n+2)/(n-2)}
\]
in the sense of \cite{MR2148873}. The hypothesis $V_g(p_0)<0$ becomes 
\[
a(p_0)<\tfrac{n-2}{4(n-1)}S_g(p_0),
\]
that is, the linear term of the equation drops below the linear term of the standard Yamabe equation at the point $p_0$. Existence of a smooth positive solution then follows from \cite[Thm.~4.1]{MR2148873} and the discussion thereafter.
\end{proof}

Combining \cref{thm:fulllocus} and \cref{thm:negdefect} we have:

\maintheoremthree*

\begin{rmk}
\cref{thm:negdefect} covers, in particular, the empty-locus subcase $V_g < 0$ on $M$. The complementary empty-locus subcase $V_g > 0$ on $M$ remains open.
\end{rmk}

\appendix

\section{Ricci curvature of a Weyl connection}\label{sec:ricci_curvature_of_a_weyl_connection}

Since the formula for the Ricci curvature of a Weyl connection does not seem to be easy to locate in the literature, we include the calculation.

Let $\nabla$ be a torsion-free connection on $TM$ with curvature tensor $R=(R^{i}_{jkl})$ and let $\Phi=(\Phi^i_{jk})$ be a $1$-form on $M$ with values in the endomorphisms of $TM$ satisfying $\Phi(v)(w)=\Phi(w)(v)$ for all $v,w \in T_pM$ and all $p \in M$. Then the torsion-free connection $\hat{\nabla}=\nabla+\Phi$ has curvature
\[
\hat{R}^i_{jkl} = R^i_{jkl} + \nabla_k \Phi^i_{lj} - \nabla_l \Phi^i_{kj} + \Phi^u_{lj} \Phi^i_{ku} - \Phi^u_{kj} \Phi^i_{lu}
\]
as follows from an elementary computation. In particular, for the Ricci curvature $\hat{R}_{jk}=\hat{R}^l_{jlk}$ of $\hat{\nabla}$ we obtain
\[
\hat{R}_{jk}=R_{jk}+\nabla_l\Phi^l_{jk}-\nabla_k\Phi^l_{jl}+\Phi^u_{jk}\Phi^l_{ul}-\Phi^u_{jl}\Phi^l_{ku}.
\]

For $(g,X)$ consider the Weyl connection 
\[
\hat{\nabla}:=\levi+g\otimes X-\sym(X^{\flat}).
\]
Take $\Phi=g\otimes X-\sym(X^{\flat})$ and $\nabla={}^g\nabla$ so that
\[
\Phi^i_{jk}=X^ig_{jk}-\delta^i_j\xi_k-\delta^i_k\xi_j,
\]
where $\xi_i=g_{ij}X^j$. Then we have
\[
\nabla_l\Phi^l_{jk}=\nabla_l\left(X^lg_{jk}-\delta^l_j\xi_k-\delta^l_k\xi_j\right)=\div_g(X) g_{jk}-\nabla_j\xi_k-\nabla_k\xi_j.
\]
Likewise, we obtain
\[
\nabla_k\Phi^l_{jl}=-n\nabla_k\xi_j
\]
as well as
\[
\Phi^u_{jk}\Phi^l_{ul}=2n\xi_j\xi_k-n|X|^2_gg_{jk}
\]
and
\[
\Phi^u_{jl}\Phi^l_{ku}=(n+2)\xi_j\xi_k-2|X|^2_gg_{jk}.
\]
From this we compute for the Ricci curvature $\hat{R}_{jk}$ of the Weyl connection $\hat{\nabla}$
\[
\begin{aligned}
\hat{R}_{jk}&=R_{jk}+\div_g(X)g_{jk}-\nabla_j\xi_k-\nabla_k\xi_j+n\nabla_k\xi_j\\
&\phantom{=}+2n\xi_j\xi_k-n|X|_g^2g_{jk}-(n+2)\xi_j\xi_k+2|X|_g^2g_{jk}\\
&=R_{jk}+(n-2)(\nabla_{(j}\xi_{k)}+\xi_j\xi_k-|X|_g^2g_{jk})+\div_g(X)g_{jk}-\tfrac{n}{2}(\d\xi)_{jk},
\end{aligned}
\]
where $R_{jk}$ denotes the Ricci curvature of $\nabla$, round brackets on indices denote symmetrisation and square brackets denote anti-symmetrisation. In particular, the symmetric and antisymmetric parts of the Ricci curvature are
\[
\hat{R}_{(jk)} = R_{jk} + (n-2)\left(\nabla_{(j}\xi_{k)} + \xi_j\xi_k - |X|_g^2g_{jk}\right) + \div_g(X) g_{jk}
\]
and
\[
\hat{R}_{[jk]} = -\tfrac{n}{2}\,(\d\xi)_{jk}.
\]
From this we compute
\[
S_{g,X}=S_g+2(n-1)\div_g(X)-(n-2)(n-1)|X|_g^2,
\]
where $S_{g,X}:=\tr_g\Ric(\hat{\nabla})$ and $S_g$ denotes the scalar curvature of $g$.

\section{Projective versus Riemannian geodesics}\label{sec:spreading}

Consider a projective manifold $(M,\proj)$ equipped with a Riemannian metric $g$. For every tangent vector $v \in TM$ which has unit length with respect to $g$, we obtain two immersed curves on $M$: the $g$-geodesic $\gamma$ passing through the basepoint of $v$ with velocity $v$ at this point and a geodesic $\varepsilon$ of some representative connection $\nabla \in \proj$ realising the (oriented) tangent direction spanned by $v$ and which is parametrised by $g$-arclength. It is natural to ask how these two curves spread apart. This is encoded in the tensor $\Phi_g=(\nabla-\levi)_0$. More precisely, we have:

\begin{lem}\label{lem:spreading}
Let $(M,\proj)$ be a projective manifold equipped with a Riemannian metric $g$, $p\in M$ and $v\in T_pM$ with $g(v,v)=1$. Let $\gamma$ be the $g$-geodesic with $\gamma(0)=p$ and $\dot\gamma(0)=v$, and let $\varepsilon$ be a geodesic of some $\nabla\in\proj$ with $\varepsilon(0)=p$ and $\dot\varepsilon(0)$ a positive multiple of $v$, parametrised by $g$-arclength $s$. Then
\[
\operatorname{dist}_g\big(\varepsilon(s),\gamma(s)\big)=\tfrac{1}{2}s^2\big|\Pi_{v^{\perp}}\big(\Phi_g(v,v)\big)\big|_g+O(s^3).
\]
\end{lem}
\begin{rmk}
Here for $v\in T_pM$ we write $\Phi_g(v,v)=\Phi_g(v)v\in T_pM$ and denote by $\Pi_{v^{\perp}}$ the $g$-orthogonal projection onto 
\[
v^{\perp}=\{u\in T_pM\,|\,g(v,u)=0\}.
\]
Furthermore, $\operatorname{dist}_g(\,\cdot\,,\cdot)$ denotes the usual metric distance induced by $g$. 
\end{rmk}
\begin{proof}[Proof of \cref{lem:spreading}]
Let $x=(x^i) : U_p \to \R^n$ denote $p$-centered geodesic normal coordinates with respect to $g$. We write $c(s):=x(\gamma(s))$ and $e(s):=x(\varepsilon(s))$ for $s$ sufficiently small. By assumption we have $c(s)=s w$ and $\dot{e}(0)=w$, where $w=\d x(v)$ and dots denote derivatives with respect to $s$. Since $x$ are geodesic normal coordinates we have $|\d x(u)|=|u|_g$ for all $u \in T_pM$, where $|\cdot |$ denotes the standard Euclidean norm on $\R^n$. Expanding in $s$ we have
\[
e(s)=e(0)+s \dot{e}(0)+\frac{1}{2}s^2 \ddot{e}(0)+O(s^3)=sw+\frac{1}{2}s^2\ddot{e}(0)+O(s^3)
\]
and hence
\[
	e(s)-c(s)=\frac{1}{2}s^2\ddot{e}(0)+O(s^3).
\]
We first show that
\begin{equation}\label{eq:spreadingacc}
\ddot{e}(0)=-\d x\left(\Pi_{v^\perp}(\Phi_g(v,v))\right).
\end{equation}
Since $\varepsilon$ is -- up to reparametrisation -- a geodesic of $\nabla$, there exists a smooth function $\alpha$ so that for all $1\leqslant i\leqslant n$
\[
	\ddot{e}^i(s)+\Gamma^i_{jk}(e(s))\dot{e}^j(s)\dot{e}^k(s)=\alpha(s) \dot{e}^i(s),
\]
where $e^i$ denotes the $i$-th component of $e$ and $\Gamma^i_{jk}$ the Christoffel symbols of $\nabla$ with respect to $x$. Since $x$ are normal coordinates for $g$, the Christoffel symbols of $\levi$ vanish at $p$, and hence $\Gamma^i_{jk}(p)$ are the components of $\Xi_g:=\nabla-\levi$ at $p$ with respect to $x$. By \eqref{eqn:glndecomp} we have $\Xi_g=\Phi_g+\frac{1}{n+1}\sym(\beta)$ with $\beta=\con(\Xi_g)$, and $\sym(\beta)(u,u)=2\beta(u)u$ for all $u \in TM$. Evaluating the above equation at $s=0$ we thus obtain
\[
\ddot{e}(0)=-\d x\left(\Phi_g(v,v)\right)+\lambda w, \qquad \text{where}\quad \lambda=\alpha(0)-\tfrac{2}{n+1}\beta(v).
\]
In order to determine $\lambda$, we use that $\varepsilon$ is parametrised by $g$-arclength, that is, $g_{ij}(e(s))\dot{e}^i(s)\dot{e}^j(s)=1$ for all $s$, where $g_{ij}$ denote the components of $g$ with respect to $x$. Differentiating the previous identity at $s=0$ and using that $g_{ij}(p)=\delta_{ij}$ and $\partial_k g_{ij}(p)=0$ in normal coordinates gives
\[
0=2\,\delta_{ij}\,\dot{e}^i(0)\ddot{e}^j(0)=2\,\langle w,\ddot{e}(0)\rangle,
\]
where $\langle\cdot,\cdot\rangle$ denotes the Euclidean inner product. Hence $\ddot{e}(0)$ is Euclidean-orthogonal to $w$, which is a Euclidean unit vector. Since $\d x : (T_pM,g(p)) \to (\R^n,\langle\cdot,\cdot\rangle)$ is a linear isometry, it follows that $\ddot{e}(0)$ is the image under $\d x$ of the component of $-\Phi_g(v,v)$ orthogonal to $v$, which is \eqref{eq:spreadingacc}.

Finally, since $g_{ij}(p)=\delta_{ij}$, it is a standard fact that $\operatorname{dist}_g(q_1,q_2)=|x(q_1)-x(q_2)|\big(1+O(r)\big)$ for $q_1,q_2$ in the $g$-ball of radius $r$ about $p$. Since $\gamma(s)$ and $\varepsilon(s)$ lie in the $g$-ball of radius $s$ about $p$, applying this with $r=s$ and using $|e(s)-c(s)|=O(s^2)$ gives
\[
\operatorname{dist}_g(\varepsilon(s),\gamma(s))=\tfrac{1}{2}s^2|\ddot{e}(0)|+O(s^3)=\tfrac{1}{2}s^2\big|\Pi_{v^\perp}(\Phi_g(v,v))\big|_g+O(s^3),
\]
where the last equality uses \eqref{eq:spreadingacc} and $|\d x(u)|=|u|_g$ for $u\in T_pM$. This proves the lemma.
\end{proof}

\providecommand{\mr}[1]{\href{http://www.ams.org/mathscinet-getitem?mr=#1}{MR~#1}}
\providecommand{\zbl}[1]{\href{http://www.zentralblatt-math.org/zmath/en/search/?q=an:#1}{zbM~#1}}
\providecommand{\arxiv}[1]{\href{http://www.arxiv.org/abs/#1}{arXiv:#1 }}
\providecommand{\doi}[1]{\href{http://dx.doi.org/#1}{DOI~#1}}
\providecommand{\href}[2]{#2}
\providecommand{\etalchar}[1]{$^{#1}$}

\end{document}